\documentclass[12pt,fleqn]{article}
\usepackage{a4}
\usepackage{latexsym,amsmath,amssymb,amsthm,setspace}
\usepackage[english]{babel}
\usepackage{color}

\newtheorem*{theorem*}{Theorem}
\newtheorem{theorem}{Theorem}
\newtheorem{lemma}{Lemma}
\newtheorem{proposition}{Proposition}
\newtheorem*{proposition*}{Proposition}
\newtheorem{conjecture}{Conjecture}
\newtheorem*{conjecture*}{Conjecture}
\newtheorem{corollary}{Corollary}
\newtheorem*{corollary*}{Corollary}

\newcommand{\Ra}{\Rightarrow}
\newcommand{\llr}{\Longleftrightarrow}
\newcommand{\ds}{\displaystyle}

\newcommand{\tL}{\tilde{L}}
\newcommand{\tp}{\tilde{p}}

\newcommand{\tq}{\tilde{q}}
\newcommand{\tr}{\tilde{r}}
\newcommand{\tQ}{\tilde{Q}}
\newcommand{\tS}{\tilde{S}}

\newcommand{\pa}{\partial}
\newcommand{\al}{\alpha}
\newcommand{\be}{\beta}
\newcommand{\ga}{\gamma}
\newcommand{\Ga}{\Gamma}
\newcommand{\de}{\delta}
\newcommand{\eps}{\epsilon}
\newcommand{\ze}{\zeta}

\newcommand{\ka}{\kappa}
\newcommand{\si}{\sigma}

\date{February 18, 2019}

\begin{document}
\setlength{\voffset}{-1.0cm}

\begin{titlepage}
\title{On a property of random walk polynomials involving Christoffel functions}

\author{\\Erik A. van Doorn$^a$ and Ryszard Szwarc$^b$\\
\\$^a$Department of Applied Mathematics, University of Twente \\
P.O. Box 217, 7500 AE Enschede, The Netherlands\\
E-mail: e.a.vandoorn@utwente.nl\\
\\$^b$Institute of Mathematics, Wroc{\l}aw University\\
pl. Grunwaldzki 2/4, 50-384 Wroc{\l}aw, Poland\\
E-mail: ryszard.szwarc@math.uni.wroc.pl}

\maketitle
\thispagestyle{empty}

\noindent{\bf Abstract.} 
Discrete-time birth-death processes may or may not have certain 
properties known as {\em asymptotic aperiodicity\/} and the {\em strong ratio
limit property\/}. In all cases known to us a suitably normalized process having
one property also possesses the other, suggesting equivalence of the two
properties for a normalized process. We show that equivalence may be
translated into a property involving Christoffel functions for a type of
orthogonal polynomials known as {\it random walk polynomials}. The prevalence
of this property -- and thus the equivalence of asymptotic aperiodicity and the
strong ratio limit property for a normalized birth-death process -- is proven under
mild regularity conditions.



\end{titlepage}

\renewcommand{\baselinestretch}{1.5} \small\normalsize

\leftmargin 25 pt
\leftmargini 25 pt

\section{Introduction}

In what follows $\mathcal{X} := \{X(n),~n=0,1,\ldots\}$ is a (discrete-time)
birth-death process on $\mathcal{N} := \{0,1,\dots\}$, with tridiagonal
matrix of one-step transition probabilities
\[
P \equiv (P_{ij})_{i,j\in\mathcal{N}} :=
\begin{pmatrix}
r_0 & p_0 & 0 & 0 & 0 & \ldots\cr
         q_1 & r_1 & p_1 & 0 & 0 & \ldots\cr
         0 & q_2 & r_2 & p_2 & 0 & \ldots\cr
         \ldots & \ldots & \ldots & \ldots & \ldots & \ldots\cr
         \ldots & \ldots & \ldots & \ldots & \ldots & \ldots\cr
\end{pmatrix}.
\]
We assume throughout that $p_j>0,~q_{j+1}>0,~r_j \geq 0,$ and (save for the
last section) that $p_j+q_j+r_j = 1$ for $j \in \mathcal{N}$, where $q_0 := 0$.
The polynomials $Q_n$ are defined by the recurrence relation
\begin{equation}
\label{recQ}
\begin{array}{l}
xQ_n(x)=q_nQ_{n-1}(x)+r_nQ_n(x)+p_nQ_{n+1}(x),\quad n > 1,\\
Q_0(x)=1,\quad  p_0Q_1(x)=x-r_0,
\end{array}
\end{equation}
so that $Q_n(1) = 1$ for all $n$.
Karlin and McGregor \cite{KM59} referred to $\mathcal{X}$ as a {\em random
walk\/} and to $\{Q_n\}$ as a sequence of {\em random walk polynomials\/}. Since
the latter terminology is rather well established (contrary to the former\/) we will
stick with it. But note that the random walk polynomials in, for example, Askey and
Ismail \cite{AI84} have $r_j=0$ for all $j$, so the present setting is more general. 

It has been shown in \cite{KM59} that the $n$-step transition probabilities
\[
P_{ij}(n) := \Pr\{X(n)=j\,|\,X(0)=i\}, \quad i,j \in \mathcal{N}, ~n\geq 0,
\]
which satisfy $P_{ij}(n) = (P^n)_{ij}$, may also be represented in the form
\begin{equation}
\label{repP}
P_{ij}(n) = \pi_j\int_{[-1,1]}x^nQ_i(x)Q_j(x)\psi(dx),
\quad i,j \in \mathcal{N}, ~n\geq 0,
\end{equation}
where
\[
\pi_0 := 1,~~ \pi_j := \frac{p_0p_1\ldots p_{j-1}}{q_1q_2\ldots q_j},
\quad j \geq 1, 
\]
and $\psi$ is the (unique) Borel measure on the real axis of total mass 1 with
respect to which the polynomials $Q_n$ are orthogonal. Moreover, supp($\psi$),
the support of the measure $\psi$, is infinite and a subset of the interval
$[-1,1]$. Adopting the terminology of \cite{DS93} we will refer to $\psi$ as a
{\it random walk measure\/}. 

The process $\mathcal{X}$ is said to have the {\em strong ratio limit property\/}
if the limits
\begin{equation}
\label{limP}
\lim_{n\to\infty} \frac{P_{ij}(n)}{P_{kl}(n)}, \quad i,j,k,l \in \mathcal{N},
\end{equation}
exist simultaneously. $\mathcal{X}$ is {\em asymptotically periodic\/} if, in
the long run, the process evolves cyclically between the even and the odd
states, and {\em asymptotically aperiodic\/} otherwise. These properties
will be discussed in more detail in Section 2. At this point we only remark
that in all cases known to us a suitably normalized process having the
strong ratio limit property is also asymptotically aperiodic, and vice versa. So
we conjecture that for a birth-death process that is normalized (in a sense to
be defined in the next section) the two properties are in fact equivalent.

It will be shown in this paper that equivalence of the strong ratio limit
property and asymptotic aperiodicity for a normalized birth-death process may
be translated into a property of random walk polynomials and the associated
measure involving {\em Christoffel functions\/}. Concretely, with $\rho_n$
denoting the $n$th Christoffel function associated with the random walk
measure $\psi$, and $\eta$ the largest point in the support of $\psi$, we have
equivalence of the two properties for the corresponding normalized birth-death
process if and only if
\begin{equation}
\label{conj1}
\lim_{n\to\infty}\frac{\int_{[-1,0)}(-x)^n\psi(dx)}{\int_{(0,1]} x^n\psi(dx)}=0
 ~~\llr~~  \lim_{n\to\infty}\frac{\rho_n(-\eta)}{\rho_n(\eta)} = 0.
\end{equation}
So our conjecture amounts to validity of \eqref{conj1}. But actually we
conjecture validity of the stronger property
\begin{equation}
\label{conj2}
\lim_{n\to\infty}\frac{\int_{[-1,0)}(-x)^n\psi(dx)}{\int_{(0,1]} x^n\psi(dx)}=
\lim_{n\to\infty}\frac{\rho_n(-\eta)}{\rho_n(\eta)},
\end{equation}
if the left-hand limit exists.
We will subsequently disclose mild conditions for \eqref{conj2} to prevail,
and hence for equivalence of the strong ratio limit property and asymptotic
aperiodicity for a normalized birth-death process.

The next section contains a number of preliminary and introductory results. 
Then, in Section \ref{con}, the conjectured property of random walk polynomials
is motivated and its relation with the associated birth-death process is
discussed. In the Sections \ref{Cnpsi} and \ref{ratio} we collect a number of
asymptotic results for the quantities featuring in the conjectured property of
random walk polynomials. Our main conclusions -- sufficient conditions for
\eqref{conj2} to be valid -- are drawn in Section \ref{res}. In the last section
the consequences of allowing $p_j+q_j+r_j \le 1$ will be examined.

\section{Preliminaries}
\label{pre}

This section contains additional information on the strong ratio limit property
and on asymptotic aperiodicity of a birth-death process. We also define the
normalization of a birth-death process referred to in the introduction, and
start off by collecting a number of relevant properties of the random walk
polynomials $Q_n$ and the measure $\psi$ with respect to which they are
orthogonal.

\subsection{Random walk polynomials and measure}

By \eqref{repP} we have
\[
r_j \equiv (P)_{jj} = \pi_j\int_{[-1,1]} xQ_j^2(x)\psi(dx),
\quad j \in \mathcal{N},
\]
so our assumption $r_j \geq 0$ implies
\begin{equation}
\label{rw}
\int_{[-1,1]} xQ_n^2(x)\psi(dx) \geq 0, \quad n \geq 0.
\end{equation}
Whitehurst \cite[Theorem 1.6]{W82} has shown that, conversely, any Borel measure
$\psi$ on the interval $[-1,1]$, of total mass 1 and with infinite support, is a
random walk measure if it satisfies \eqref{rw} (see also \cite[Theorem 1.2]{DS93}).

Obviously $P_{ij}(0)= \delta_{ij}$ (Kronecker's delta), so, letting
\begin{equation}
\label{pn}
p_n(x) := \sqrt{\pi_n}Q_n(x), \quad n \geq 0,
\end{equation}
\eqref{repP} leads to
\[
\int_{[-1,1]} p_i(x)p_j(x)\psi(dx) = \delta_{ij}, \quad i,j \geq 0,
\]
that is, $\{p_n\}$ constitutes the sequence of {\it orthonormal\/} polynomials
with respect to the random walk measure $\psi$. Writing $p_n(x) = \ga_nx^n+\dots$
we note for future reference that
\begin{equation}
\label{gamma}
\ga_n^{-2} = \prod_{i=1}^n p_{i-1}q_i, \quad n\geq 1.
\end{equation}
The {\it Christoffel functions\/} $\rho_n$ associated with $\psi$ are defined
by
\begin{equation}
\rho_n(x) := \left\{\sum_{j=0}^{n-1}p_j^2(x)\right\}^{-1}, \quad n \geq 1.
\end{equation}
A direct relation between the measure $\psi$ and its Christoffel functions
is given by the classic result (Shohat and Tamarkin \cite[Corollary 2.6]{ST63})
\begin{equation}
\label{ST}
\lim_{n\to\infty}\rho_n(x) = \psi(\{x\}), \quad x \in \mathbb{R}.
\end{equation}

Of particular interest to us is $\eta := \sup \mbox{supp}(\psi)$, the largest point
of the support of the measure $\psi$, which may also be characterized in terms of
the polynomials
$Q_n$ by
\begin{equation}
\label{Qpos}
x \geq \eta ~~\llr ~~ Q_n(x) > 0 \mbox{~~for all~~} n \geq 0
\end{equation}
(see, for example, Chihara \cite[Theorem II.4.1]{C78}). Evidently, \eqref{rw}
already implies $\eta > 0$, but it can actually be shown (see, for example,
\cite[Corollary 2 to Theorem IV.2.1]{C78}) that
\begin{equation}
\label{loweta}
 0 \leq r_j < \eta \leq 1, \quad j \in\mathcal{N}.
\end{equation}
Letting $\ze:=\inf \mbox{supp}(\psi)$ we also have
\[
\inf_j \{r_j+r_{j+1}\}\leq \ze+\eta \leq \sup_j\{r_j+r_{j+1}\}, \quad j \in\mathcal{N},
\]
by \cite[Lemma 2.3]{DS95a}. It follows that 
\begin{equation}
\label{infsupp}
\ze \geq -\eta,
\end{equation}
and hence supp$(\psi) \subset [-\eta,\eta]$. Moreover, the counterpart of
\eqref{Qpos} (obtained from \eqref{Qpos} by considering, instead of $Q_n(x)$,
the polynomials $(-1)^nQ_n(-x)$) gives us
\begin{equation}
\label{Qneg}
x \leq \ze ~~\llr ~~ (-1)^nQ_n(x) > 0 \mbox{~~for all~~} n \geq 0.
\end{equation}

The recurrence relations \eqref{recQ} imply the {\em Christoffel-Darboux\/}
identity
\begin{equation}
\label{CD}
p_n\pi_n(Q_n(x)Q_{n+1}(y)-Q_n(y)Q_{n+1}(x))=(y-x)\sum_{j=0}^n\pi_jQ_j(x)Q_j(y)
\end{equation}
(see, for example, \cite[Theorem I.4.5]{C78}), whence, by \eqref{Qpos},
\begin{equation}
\eta \leq x < y ~~\Ra~~ Q_n(x)Q_{n+1}(y) > Q_n(y)Q_{n+1}(x) > 0
\mbox{~~for all~~} n \geq 0.
\end{equation}
Since $Q_n(1)=1$ for all $n$ this leads in particular to
\begin{equation}
\label{monQ}
\eta \leq x < 1 ~~\Ra~~  0 < Q_{n+1}(x) < Q_n(x) < Q_0(x)=1 \mbox{~~for all~~} n \geq 1.
\end{equation}

The measure $\psi$ is symmetric about $0$ if (and only if) the process
$\mathcal{X}$ is {\em periodic\/}, that is, if $r_j=0$ for all $j$ (see
\cite[p.~69]{KM59}). Evidently, the process will evolve cyclically between the even
and the odd states if it is periodic. The process is {\em aperiodic\/} if it is not
periodic. Whitehurst \cite[Theorem 5.2]{W78} has shown that
\begin{equation}
\label{whit}
\mathcal{X} \mbox{~is~aperiodic}~~\Ra~~
\int_{[-\eta,\eta]}\frac{\psi(dx)}{\eta+x} < \infty,
\end{equation}
so that in particular $\psi(\{-\eta\}) = 0$ if $\mathcal{X}$ is aperiodic.
It will also be useful to note from \eqref{recQ} that
\begin{equation}
\label{sym}
\mathcal{X} \mbox{~is~periodic}~~\llr~~(-1)^nQ_n(-x) = Q_n(x), \quad n \geq 0.
\end{equation}

We now introduce the normalization of the process $\mathcal{X}$ referred to in
the Introduction. Namely, letting $\tq_0:=0$ and
\begin{equation}
\label{tbd}
\tp_j := \frac{Q_{j+1}(\eta)}{Q_j(\eta)}\frac{p_j}{\eta},~~
\tr_j := \frac{r_j}{\eta},~~
\tq_{j+1} := \frac{Q_j(\eta)}{Q_{j+1}(\eta)}\frac{q_{j+1}}{\eta},
\quad j \in \mathcal{N},
\end{equation}
it follows from \eqref{recQ} and \eqref{Qpos} that $\tp_j>0,~\tq_{j+1}>0$,
$\tr_j \geq 0$, and $\tp_j+\tq_j+\tr_j=1$, so that the parameters $\tp_j,~\tq_j$
and $\tr_j$ may be interpreted as the one-step transition probabilities of a
birth-death process $\mathcal{\tilde{X}}$ on $\mathcal{N}$, the
{\em normalized\/} version of $\mathcal{X}$. Note that $\mathcal{\tilde{X}}$
is periodic if and only if $\mathcal{X}$ is periodic. Since $Q_n(1)=1$ for all $n$
we have $\mathcal{\tilde{X}} = \mathcal{X}$ if (and only if) $\eta=1$. By
\cite[Appendix 2]{DS95a} the random walk polynomials $\tQ_n$ and measure
$\tilde{\psi}$ associated with the process $\mathcal{\tilde{X}}$ may be
expressed as
\begin{equation}
\label{tQ}
\tQ_n(x) = \frac{Q_n(\eta x)}{Q_n(\eta)}, \quad n \geq 0.
\end{equation}
and
\begin{equation}
\label{tpsi}
\tilde{\psi}([-1,x]) = \psi([-\eta,x\eta]), \quad -1 \leq x \leq 1,
\end{equation}
respectively. Consequently,
\[
\tilde{\ze} :=\inf\mbox{supp}(\tilde{\psi}) = \frac{\ze}{\eta} \geq -1 
~~\mbox{and}~~ \tilde{\eta} := \sup \mbox{supp}(\tilde{\psi}) = 1.
\]
So normalizing $\mathcal{X}$ amounts to stretching the support of the associated
measure such that its largest point becomes $1$.

We know from \cite[Lemma 2.1]{D18a} that $(-1)^n\tQ_n(-1)$ is increasing, and
strictly increasing for $n$ sufficiently large, if $\tr_j>0$ for some $j\in
\mathcal{N}$, that is, if $\mathcal{\tilde{X}}$ is aperiodic.  It follows that
$|Q_n(\eta)/Q_n(-\eta)|$ is decreasing, and strictly decreasing for $n$
sufficiently large, if $\mathcal{X}$ is aperiodic. Since, by \eqref{sym},
$(-1)^n\tQ_n(-1) = \tQ_n(1) = 1$ for all $n$ if $\mathcal{X}$ is periodic, we
can conclude the following.
\begin{lemma}
\label{bounds}
If $\mathcal{X}$ is periodic then $Q_n^2(\eta)/Q_n^2(-\eta)=1$ for all $n$. If
$\mathcal{X}$ is aperiodic then $Q_n^2(\eta)/Q_n^2(-\eta)$ is decreasing and
tends to a limit satisfying
\[
0 \leq \lim_{n\to\infty}\frac{Q_n^2(\eta)}{Q_n^2(-\eta)} < 1.
\]
\end{lemma}
\vspace{.3cm}
In view of \eqref{pn} this lemma tells us that the ratio $p_n^2(\eta)/p_n^2(-\eta)$
tends to a limit as $n\to\infty$, while, by \eqref{ST} and \eqref{whit},
\[
\mathcal{X} \mbox{~is~aperiodic}~~\Ra~~\lim_{n\to\infty}\frac{1}{\rho_n(-\eta)} =
\sum_{j=0}^{\infty}p_j^2(-\eta)= \infty.
\]
Applying the Stolz-Ces\`aro theorem therefore leads to the conclusion that, as
$n\to\infty$, the ratio $\rho_n(-\eta)/\rho_n(\eta)$ tends to a limit satisfying
\begin{equation}
\label{rhop}
\lim_{n\to\infty}\frac{\rho_n(-\eta)}{\rho_n(\eta)}=
\lim_{n\to\infty}\frac{p_n^2(\eta)}{p_n^2(-\eta)}
\end{equation}
if $\mathcal{X}$ is aperiodic. But \eqref{rhop} is obviously also valid if
$\mathcal{X}$ is periodic (both limits then being one), so we have the following
result.
\begin{proposition}
\label{rhoQ}
If $\mathcal{X}$ is periodic then $\rho_n(-\eta)/\rho_n(\eta) = 
Q_n^2(\eta)/Q_n^2(-\eta) = 1$ for all $n$. If $\mathcal{X}$ is aperiodic then
$\rho_n(-\eta)/\rho_n(\eta)$ tends, as $n\to\infty$, to a limit satisfying
\[
0 \leq \lim_{n\to\infty}\frac{\rho_n(-\eta)}{\rho_n(\eta)}=
\lim_{n\to\infty}\frac{Q_n^2(\eta)}{Q_n^2(-\eta)} < 1.
\]
\end{proposition}
With $\tilde{\rho}_n$ denoting the Christoffel functions associated with the
normalized process $\mathcal{\tilde{X}}$ it follows readily from \eqref{pn},
\eqref{tQ} and \eqref{rhop} that
\begin{equation}
\label{trho}
\lim_{n\to\infty}\frac{\rho_n(-\eta)}{\rho_n(\eta)}=
\lim_{n\to\infty}\frac{\tilde{\rho}_n(-1)}{\tilde{\rho}_n(1)},
\end{equation}
so in studying the asymptotic behaviour of the ratio
$\rho_n(-\eta)/\rho_n(\eta)$ it is no restriction to assume $\eta=1$.

We will see in the next subsections that Proposition \ref{rhoQ} enables us to
establish a link between the Christoffel functions associated with a sequence of
random walk polynomials and probabilistic properties of the normalized version
of the corresponding birth-death process.

\subsection{Strong ratio limit property}
\label{srlp}

The strong ratio limit property (SRLP) was introduced in the setting of
discrete-time Markov chains on a countable state space by Orey \cite{O61} and
Pruitt \cite{P65}, but the problem of finding conditions for the limits
\eqref{limP} to exist in the more restricted setting of discrete-time
birth-death processes had been considered before in \cite{KM59}. For more
information on the history of the problem we refer to \cite{DS95b} and
\cite{K95}.

A necessary and sufficient condition for the process $\mathcal{X}$ to possess
the SRLP is known in terms of the associated random walk measure $\psi$.
Namely, letting
\begin{equation}
\label{Cn}
C_n(\psi) := \frac{\int_{[-1,0)}(-x)^n\psi(dx)}{\int_{(0,1]} x^n\psi(dx)},
\quad n \geq 0,
\end{equation}
\cite[Theorem 3.1]{DS95b} tells us the following.
\begin{theorem}
\label{srlpth}
The process $\mathcal{X}$ has the SRLP if and only if $\ds\lim_{n\to\infty}C_n(\psi)
= 0$, in which case
\[
\lim_{n\to\infty} \frac{P_{ij}(n)}{P_{kl}(n)} =
\frac{\pi_jQ_i(\eta)Q_j(\eta)}{\pi_lQ_k(\eta)Q_l(\eta)}, \quad i,j,k,l \in\mathcal{N}.
\]
\end{theorem}
Note that the denominator in \eqref{Cn} is positive since $\eta>0$, so that
$C_n(\psi)$ exists and is nonnegative for all $n$. Moreover, in view of \eqref{tpsi}
we clearly have
\begin{equation}
\label{tC}
C_n(\psi)=C_n(\tilde{\psi}), \quad n \geq 0,
\end{equation}
so normalization does not affect prevalence of the SRLP.

If $\mathcal{X}$ is periodic then $P_{ij}(n) =0$ if $n+i+j$ is odd, as a consequence
of \eqref{repP} and \eqref{sym}. Hence the limits in \eqref{limP} do not exist, which
is reflected by the fact that $C_n(\psi) = 1$ for all $n$ in this case. So aperiodicity is
necessary for $\mathcal{X}$ to have the SRLP. A sufficient condition for
$\mathcal{X}$ to have the SRLP is implied by \cite[Theorem 3.2]{DS95b}, which
states that
\begin{equation}
\label{QC}
\lim_{n\to\infty} \left|\frac{Q_n(\eta)}{Q_n(-\eta)}\right| = 0
~~\Ra~~ \lim_{n\to\infty} C_n(\psi) = 0.
\end{equation}
The reverse implication is conjectured in \cite{DS95b} to be valid as well.
We can actually establish a result that is stronger than \eqref{QC}.
\begin{lemma}
\label{supC}
We have
\[
0\leq\limsup_{n\to\infty} C_n(\psi) \leq
\lim_{n\to\infty} \frac{Q^2_n(\eta)}{Q^2_n(-\eta)}.
\]
\end{lemma}
\begin{proof}\rm 
The first inequality is obvious since $C_n(\psi)\geq 0$ for all $n$. If
$\mathcal{X}$ is periodic, then, by \eqref{sym} and the fact that $\psi$ is
symmetric, both sides of the second inequality are one, so in the remainder of
this proof we will assume that $\mathcal{X}$ is aperiodic. Let
\[
c_1:= \limsup_{n\to\infty}C_{2n}(\psi), \quad c_2:= \limsup_{n\to\infty}C_{2n+1}(\psi),
\]
and
\begin{equation}
\label{Lf}
L_n(f,\psi):= \frac{\int_{[-\eta,\eta]}x^n f(x)\psi(dx)}{\int_{[-\eta,\eta]}x^n \psi(dx)},
\quad n\geq 0.
\end{equation}
In view of the representation formula \eqref{repP} the denominator in \eqref{Lf}
equals $P_{00}(n)$ and is therefore nonnegative for all $n$. But, $\mathcal{X}$
being aperiodic, we must have $P_{00}(n)>0$ for $n$ sufficiently large so the
denominator is actually positive for $n$ sufficiently large.
Choosing a subsequence $n_k$ of the positive integers such that $C_{2n_k}(\psi)
\to c_1$ as $k\to\infty$, we have, by \cite[Lemmas 3.1 and 3.2]{DS95b},
\[
\lim_{k\to\infty} L_{n_k}(Q_jQ_{j+1},\psi) =
\frac{Q_j(\eta)Q_{j+1}(\eta)+c_1Q_j(-\eta)Q_{j+1}(-\eta)}{1+c_1}.
\]
Since, by the representation formula \eqref{repP} again, $L_{n}(Q_jQ_{j+1},\psi)
\geq 0$ for all $n$, the limit must be nonnegative. Moreover, by \eqref{infsupp}
and \eqref{Qneg} we have $(-1)^n Q_n(-\eta) > 0$ for all $n \geq 0$, so that
$Q_j(-\eta)Q_{j+1}(-\eta)<0$. Hence 
\[
c_1 \leq -\frac{Q_j(\eta)Q_{j+1}(\eta)}{Q_j(-\eta)Q_{j+1}(-\eta)}, \quad j \geq 0,
\]
so that
\[
c_1 \leq \lim_{j\to\infty}-\frac{Q_j(\eta)Q_{j+1}(\eta)}{Q_j(-\eta)Q_{j+1}(-\eta)}
=\lim_{n\to\infty}\frac{Q^2_n(\eta)}{Q^2_n(-\eta)}.
\]

Turning to $c_2$ we first note that $0\leq c_2<1$ by \cite[Lemma 3.3]{DS95b}.
Next proceeding in a similar way as before, by considering $L_{n_k}(Q_j^2,\psi)$ with
$n_k$ a subsequence of the integers such that $C_{2n_k+1}(\psi)\to c_2$, we obtain
\[
\lim_{k\to\infty} L_{n_k}(Q^2_j,\psi) =
\frac{Q^2_j(\eta)-c_2Q^2_j(-\eta)}{1-c_2},
\]
so that
\[
c_2 \leq \lim_{n\to\infty} \frac{Q^2_n(\eta)}{Q^2_n(-\eta)},
\]
which completes the proof.
\end{proof}

In view of Proposition \ref{rhoQ} we can thus state the following.
\begin{theorem}
\label{supCrho}
If $\mathcal{X}$ is aperiodic then
\[
0 \leq \limsup_{n\to\infty} C_n(\psi) \leq
\lim_{n\to\infty}\frac{\rho_n(-\eta)}{\rho_n(\eta)} < 1.
\]
\end{theorem}

It has recently been shown in \cite[Lemma 2.1]{D18a} that
\begin{equation}
\label{RQ}
\sum_{j=0}^\infty \frac{1}{p_j\pi_j} \sum_{k=0}^j r_k\pi_k =\infty
~~\llr~~ \lim_{n\to\infty} (-1)^nQ_n(-1) = \infty,
\end{equation}
while it follows from \cite[Corollary 3.2 and Lemma 3.3]{D18a} that
\begin{equation}
\label{QQ}
\lim_{n\to\infty} |Q_n(-1)| = \infty
~~\Ra~~ \lim_{n\to\infty} \left|\frac{Q_n(\eta)}{Q_n(-\eta)}\right| = 0.
\end{equation}
Hence, by Proposition \ref{rhoQ},
\begin{equation}
\label{Rrho}
\sum_{j=0}^\infty \frac{1}{p_j\pi_j} \sum_{k=0}^j r_k\pi_k =\infty
~~\Ra~~ \lim_{n\to\infty} \frac{\rho_n(-\eta)}{\rho_n(\eta)}= 0,
\end{equation}
which, in view of Theorem \ref{supCrho}, gives us a sufficient condition for
the SRLP directly in terms of the parameters of the process.
The condition is not necessary since \cite[Example 4.1]{D18a} provides a
counterexample to the reverse implication in \eqref{QQ}.

\subsection{Asymptotic aperiodicity}
\label{AA}

A discrete-time Markov chain on $\mathcal{N}$ may, in the long run, evolve
cyclically through a number of sets constituting a partition of $\mathcal{N}$.
The maximum number of sets involved in this cyclic behaviour is called the
{\em asymptotic period\/} of the chain, and the chain is said to be
{\em asymptotically aperiodic\/} if such cyclic behaviour does not occur, in
which case we also say that the asymptotic period equals one. The {\em
asymptotic period\/} of a Markov chain may be larger than its {\em period\/}.
For rigorous definitions and developments we refer to \cite{D18b}, where it is
also shown that in the specific setting of a birth-death process
the asymptotic period equals either one, or two, or infinity. Precise conditions
for these values to prevail are given as well. In particular, \cite[Theorem 12]{D18b}
tells us the following.
\begin{theorem}
\label{AR}
The process $\mathcal{X}$ is asymptotically aperiodic if and only if 
\begin{equation}
\label{ARsum}
\sum_{j=0}^\infty \frac{1}{p_j\pi_j} \sum_{k=0}^j r_k\pi_k =\infty.
\end{equation}
\end{theorem}
Note that \eqref{ARsum} is precisely the sufficient condition for prevalence of
the SRLP derived in the previous subsection.

Letting
\begin{equation}
\label{Ln}
L_n := \sum_{j=0}^n \frac{1}{p_j\pi_j}, \quad 0\leq n\leq\infty,
\end{equation}
it follows from Theorem \ref{AR} that
\begin{equation}
\label{recAA}
\mathcal{X} ~~\mbox{is~aperiodic~and~} L_\infty=\infty
~~\Ra~~ \mathcal{X} \mbox{~is~asymptotically~aperiodic}.
\end{equation}
So, recalling from \cite{KM59} that
\begin{equation}
\label{classification}
\mathcal{X} ~\mbox{is}~
\left\{
\begin{array}{l@{}l}
\mbox{recurrent} &~~\llr~~ L_\infty = \infty\\
\mbox{transient} &~~\llr~~ L_\infty < \infty,
\end{array}
\right.
\end{equation}
and noting the obvious fact that asymptotic aperiodicity implies aperiodicity,
we conclude that for a recurrent process {\em aperiodicity\/} and
{\em asymptotic aperiodicity\/} are equivalent.  The study of asymptotic
aperiodicity is therefore relevant in particular for transient processes.

Another sufficient condition for asymptotic aperiodicity is obtained by
observing that
\begin{equation}
\label{ineqr}
\sum_{j=0}^n \frac{1}{p_j\pi_j} \sum_{k=0}^j r_k\pi_k 
\geq \sum_{j=0}^n \frac{r_j}{p_j},
\end{equation}
so that
\begin{equation}
\label{rjpj}
\sum_{j=0}^\infty\frac{r_j}{p_j} = \infty ~~\Ra~~ \mathcal{X}
\mbox{~is~asymptotically~aperiodic}.
\end{equation}

Now turning to the normalized version $\mathcal{\tilde{X}}$ of $\mathcal{X}$
we observe from the analogues for $\mathcal{\tilde{X}}$ of \eqref{RQ} and
Theorem \ref{AR} that
\[
\mathcal{\tilde{X}} \mbox{~is~asymptotically~aperiodic~} 
~~\llr~~ \lim_{n\to\infty} (-1)^n\tQ_n(-1) = \infty,
\]
which, by \eqref{tQ} and Proposition \ref{rhoQ}, may be formulated as
\begin{equation}
\label{tXrho}
\mathcal{\tilde{X}} \mbox{~is~asymptotically~aperiodic~} 
~~\llr~~ \lim_{n\to\infty} \frac{\rho_n(-\eta)}{\rho_n(\eta)}= 0.
\end{equation}
With Theorem \ref{AR} it now follows that \eqref{Rrho} may be translated into
\begin{equation}
\label{XtX}
\mathcal{X} \mbox{~is~asymptotically~aperiodic~} ~~\Ra~~
\mathcal{\tilde{X}} \mbox{~is~asymptotically~aperiodic}, 
\end{equation}
but we emphasize again that the reverse implication is not valid.

\section{Conjecture}
\label{con}

In view of \eqref{Rrho} and the Theorems \ref{srlpth}, \ref{supCrho}
and \ref{AR} the birth-death process $\mathcal{X}$ has the SRLP if it is
asymptotically aperiodic. But, bearing in mind that the reverse implication in
\eqref{Rrho} does not hold, the two properties are definitely {\em not\/} equivalent.
However, if, instead of $\mathcal{X}$, we consider the {\em normalized\/}
process $\mathcal{\tilde{X}}$, then $|\tQ_n(\tilde{\eta})/\tQ_n(-\tilde{\eta})|
=|1/\tQ_n(-1)|$, so that the reverse implication in \eqref{QQ} -- and hence in
\eqref{Rrho} -- is trivially true. In all cases known to us a normalized process
having the SRLP is asymptotically aperiodic, so we conjecture that
$\mathcal{\tilde{X}}$ is in fact asymptotically aperiodic if it has the SRLP, which,
by Theorem \ref{srlpth}, \eqref{tC} and \eqref{tXrho}, amounts to the following.

\begin{conjecture}
\label{c1}
We have
\begin{equation}
\label{con1}
\lim_{n\to\infty} C_n(\psi) = 0 ~~\Ra~~
\lim_{n\to\infty} \frac{\rho_n(-\eta)}{\rho_n(\eta)}= 0.
\end{equation}
\end{conjecture}

Recall that, by Proposition \ref{rhoQ}, the limit on the right-hand side exists,
and that, by Theorem \ref{supCrho}, the right-hand side of \eqref{con1}
implies the left-hand side. Note also that \eqref{con1} is equivalent
to the conjecture already put forward in \cite{DS95b}. Actually, as announced in
the introduction, we venture to state the following, stronger conjecture.
\begin{conjecture}
\label{c2}
If $C_n(\psi)$ tends to a limit as $n\to\infty$, then
\begin{equation}
\label{con2}
\lim_{n\to\infty}
C_n(\psi) = \lim_{n\to\infty} \frac{\rho_n(-\eta)}{\rho_n(\eta)}.
\end{equation}
\end{conjecture}

In what follows we will verify Conjecture 2 -- and hence Conjecture 1 -- under
some mild regularity conditions. But before drawing our conclusions in Section
\ref{res}, we collect some asymptotic properties of $C_n(\psi)$ in the next
section and study the asymptotic behaviour of the ratio $\rho_n(-\eta)/\rho_n(\eta)$
in Section \ref{ratio}. 

\section{Asymptotic results for $C_n(\psi)$}
\label{Cnpsi}

By definition of $C_n(\psi)$ we obviously have $C_n(\psi)=0$ for all $n$ if
$\ze\geq 0$. Moreover, if $-\eta<\ze<0$ then, for $0<\eps<\eta+\ze$,
\[
C_n(\psi) = \frac{\int_{[\ze,0)}(-x)^n\psi(dx)}{\int_{(0,\eta]} x^n\psi(dx)}
\leq \frac{(-\ze)^n}{(\eta-\eps)^n\psi([\eta-\eps,\eta])} \to 0 ~~\mbox{as}~~n\to\infty.
\]
Finally, if $\ze=-\eta$ we have, for $0<\eps<\eta$,
\[
\frac{\int_{(0,\eta-\eps]}x^n\psi(dx)}{\int_{[\eta-\eps,\eta]} x^n\psi(dx)}\leq
\frac{(\eta-\eps)^n}{(\eta-\eps/2)^n\psi([\eta-\eps/2,\eta])}\to 0
~~\mbox{as}~~n\to\infty,
\]
while
\[
\frac{\int_{[-\eta+\eps,0)}(-x)^n\psi(dx)}{\int_{[-\eta,-\eta+\eps]}(-x)^n\psi(dx)}\leq
\frac{(\eta-\eps)^n}{(\eta-\eps/2)^n\psi([-\eta,-\eta+\eps/2])}\to 0
~~\mbox{as}~~n\to\infty.
\]
With these results we readily obtain the next proposition, which extends
\cite[Lemma 3.5]{DS95b}.

\begin{proposition}
\label{eps}
If $\ze>-\eta$ then $\ds\lim_{n\to\infty}C_n(\psi)=0$.
If $\ze=-\eta$ then we have for any $\eps\in (0,\eta)$,
\vspace{-.3cm}
\begin{equation}
\label{Cint}
\ds\limsup_{n\to\infty} C_n(\psi) = \ds\limsup_{n\to\infty}
\frac{\int_{[-\eta,-\eta+\eps]}(-x)^n\psi(dx)}{\int_{[\eta-\eps,\eta]} x^n\psi(dx)},
\end{equation}
and a similar result with $\limsup$ replaced by $\liminf$.
\end{proposition}

As an aside we note that the first statement of this proposition follows also from
Theorem \ref{ze>-eta} in the next section and Theorem \ref{supCrho}.

\begin{corollary}
\label{epscor}
Let $0<\eps<\eta$. Then $\lim_{n\to\infty} C_n(\psi)$ exists if and only if the ratio of
integrals in \eqref{Cint} tends to a limit as $n\to\infty$, in which case the two limits
are equal.
\end{corollary}

This corollary and \eqref{whit} imply in particular that $\lim_{n\to\infty}
C_n(\psi) = 0$ if $\mathcal{X}$ is aperiodic and $\psi(\{\eta\})>0$. But this result
is encompassed in the next proposition.

\begin{proposition}
We have
\begin{equation}
\label{boundsC}
\begin{array}{l}
\ds\liminf_{\eps\downarrow 0}\frac{\psi([-\eta,-\eta+\eps])}{\psi([\eta-\eps,\eta])} \leq
\ds\liminf_{n\to\infty} C_n(\psi) \leq\\
\hspace{4.7cm}\ds\limsup_{n\to\infty} C_n(\psi) \leq
\ds\limsup_{\eps\downarrow 0}\frac{\psi([-\eta,-\eta+\eps])}{\psi([\eta-\eps,\eta])}.
\end{array}
\end{equation}
\end{proposition}
\vspace{.0cm}
\begin{proof}\rm 
The result is obviously true if $\psi$ is symmetric about $0$ (that is, if
$\mathcal{X}$ is periodic) or, by Proposition \ref{eps}, if $\ze>-\eta$. Moreover,
if $\mathcal{X}$ is aperiodic, $\ze = -\eta$ and $\psi(\{\eta\}) > 0$ then, by
\eqref{whit} and Corollary \ref{epscor}, all components of the inequalities \eqref{boundsC}
are zero. In the remainder of the proof we will therefore assume that $\mathcal{X}$
is aperiodic, $\ze = -\eta$ and $\psi(\{\eta\}) = \psi(\{-\eta\}) = 0$. Now let $c$ be
such that
\[
c>L:=\limsup_{\eps\downarrow 0}\frac{\psi([-\eta,-\eta+\eps])}{\psi([\eta-\eps,\eta])}.
\]
Then there exists an $\eps$, $0 <\eps < \eta$, such that 
\begin{equation}
\label{c}
\psi([-\eta,-\eta+x])\leq c\psi([\eta-x,\eta]), \quad 0 < x \leq \eps.
\end{equation}
Next defining
\begin{equation}
\label{Psi}
\Psi(x) := \left\{
\begin{array}{l@{}l}
0  &\mbox{if}~~ x < -\eta\\
\psi([-\eta,x]) \qquad &\mbox{if}~ -\eta \leq x \leq \eta\\
1  &\mbox{if}~~ x > \eta.
\end{array}
\right.
\end{equation}
integration by parts of the relevant Stieltjes integrals gives us, for all
$n$,
\[
\begin{array}{l@{}l}
\ds\int_{[\eta-\eps,\eta]}x^n\psi(dx)~  & = \ds\int_{[\eta-\eps,\eta]} x^n d\Psi(x)\\
&\ds= \eta^n-(\eta-\eps)^n\Psi(\eta-\eps-) -n\int_{\eta-\eps}^\eta x^{n-1}\Psi(x)dx \\
&\ds=  n\int_{\eta-\eps}^\eta[1-\Psi(x)]x^{n-1}dx + (\eta-\eps)^n[1-\Psi(\eta-\eps-)]\\
&\ds= n\int_{\eta-\eps}^\eta \psi([x,\eta])x^{n-1}dx + (\eta-\eps)^n\psi([\eta-\eps,\eta]),
\end{array}
\]
while
\[
\begin{array}{l@{}l}
\ds\int_{[-\eta,-\eta+\eps]}(-x)^n\psi(dx)~  & = \ds\int_{[\eta-\eps,\eta]} x^n d(1-\Psi(-x))\\
&\ds=  n\int_{\eta-\eps}^\eta \Psi(-x) x^{n-1}dx + (\eta-\eps)^n \Psi(-\eta+\eps)]\\
&\ds= n\int_{\eta-\eps}^\eta \psi([-\eta,-x])x^{n-1}dx + (\eta-\eps)^n\psi([-\eta,-\eta+\eps])\\
&\ds\leq cn\int_{\eta-\eps}^\eta \psi([x,\eta])x^{n-1}dx + c(\eta-\eps)^n\psi([\eta-\eps,\eta]),
\end{array}
\]
where we have used \eqref{c} in the last step. It follows that
\[
\limsup_{n\to\infty} \frac{\int_{[-\eta,-\eta+\eps]}(-x)^n\psi(dx)}
{\int_{[\eta-\eps,\eta]} x^n\psi(dx)} \leq c,
\]
and since $c$ can be chosen arbitrarily close to $L$, the right-hand inequality in
\eqref{boundsC} follows by Proposition \ref{eps}.
The left-hand inequality is proven similarly.
\end{proof}

In combination with Theorem \ref{supCrho} this proposition leads to the following.

\begin{theorem}
\label{thmC}
If $\mathcal{X}$ is aperiodic we have
\[
\lim_{n\to\infty} C_n(\psi) =
\lim_{\eps\downarrow 0}\frac{\psi([-\eta,-\eta+\eps])}{\psi([\eta-\eps,\eta])}<1,
\]
if the second limit exists.
\end{theorem}

With a view to the analysis in Section \ref{res} we will employ this theorem
to obtain a limit result in a more specific situation. Concretely, we consider the
condition \\
$(i)$~~~ $\Psi$ is continuously differentiable on $\mathbb{R}$ and $\Psi'(x)>0$
for $x\in (-\eta,\eta)$,\\
where $\Psi$ denotes the function defined in \eqref{Psi}.
Note that this condition implies $\Psi(-\eta+) = 0$,
$\Psi(\eta-) = 1$ and also $\Psi'(-\eta+) = \Psi'(\eta-) = 0$.
If condition $(i)$ prevails we let
\begin{equation}
\label{albe}
\begin{array}{l}
\ds\al := \sup\{a: \lim_{x\uparrow\eta}~(\eta-x)^{-a} \Psi'(x) = 0\},\\ 
\ds\be := \sup\{b: \lim_{x\downarrow -\eta}~(\eta+x)^{-b} \Psi'(x) = 0\},
\end{array}
\end{equation}
so that $\al$ and $\be$ are nonnegative (but possibly infinity). A second condition is\\
$(ii)$~~ $\al$ and $\be$ are finite.\\
Finally, if conditions $(i)$ and $(ii)$ prevail we define
\begin{equation}
\label{w}
w(x) := (\eta -x)^{-\al}(\eta +x)^{-\be}\Psi'(x), \quad -\eta < x < \eta,
\end{equation}
so that $w(x)>0$ for $x\in (-\eta,\eta)$. A third condition is\\
$(iii)$~ the limits $w(-\eta+)$ and $w(\eta-)$ exist and are finite, and $w(\eta-)>0$.
\begin{theorem}
\label{Cw}
If $\mathcal{X}$ is aperiodic and  the corresponding measure $\psi$
satisfies the conditions $(i)$, $(ii)$ and $(iii)$ above, then $0 < \al\leq\be$ and
\[
\lim_{n\to\infty} C_n(\psi) =
\left\{
\begin{array}{l@{}l}
0 &\mbox{if}~~ \al<\be\\
\ds\frac{w(-\eta+)}{w(\eta-)} \quad &\mbox{if}~~ \al=\be.
\end{array}
\right.
\]
\end{theorem}
\begin{proof}
We must have $\al > 0$, since $\al=0$ would imply $w(\eta-) = 
(2\eta)^{-\be}\Psi'(\eta-) = 0$. Further, since $\Psi$ is continuously differentiable
we may apply l'H{\^ o}pital's rule to conclude that
\begin{equation}
\label{Hop}
\begin{array}{l@{}l}
\ds\lim_{\eps\downarrow 0}\frac{\psi([-\eta,-\eta+\eps])}{\psi([\eta-\eps,\eta]}~ &=
 \ds\lim_{\eps\downarrow 0}\frac{\Psi(-\eta+\eps)}{1-\Psi(\eta-\eps)}\\
&= \ds\lim_{\eps\downarrow 0}\frac{\Psi'(-\eta+\eps)}{\Psi'(\eta-\eps)}
= \ds\frac{(2\eta)^{\al-\be}}{w(\eta-)} \lim_{\eps\downarrow 0} \eps^{\be-\al}w(-\eta+\eps),
\end{array} 
\end{equation}
if the limit on the right exists. By definition of $\be$ this limit is zero if $\al<\be$,
while it obviously equals $w(-\eta+)$ if $\al=\be$.
Finally, if $\al>\be$ the right-hand limit in \eqref{Hop} is infinity, which, however,
would contradict Theorem \ref{thmC}. So we must have $\al\leq\be$. The result now
follows from Theorem \ref{thmC}.  
\end{proof}

Note that $w(-\eta+)=0$ if $\ze>-\eta$, so the theorem is consistent with Proposition
\ref{eps}.

\section{Asymptotic results for $\rho_n(-\eta)/\rho_n(\eta)$}
\label{ratio}

Formulating \eqref{RQ} and Proposition \ref{rhoQ} in terms of the normalized
process $\mathcal{\tilde{X}}$ (recall that $\tilde{\eta}=1$), and translating the
results with the help of \eqref{tbd} and \eqref{trho} in terms of quantities
related to the original process $\mathcal{X}$, leads to the next result.
\begin{lemma}
\label{rho0}
We have
\begin{equation}
\label{rho0eq}
\lim_{n\to\infty}\frac{\rho_n(-\eta)}{\rho_n(\eta)} = 0 ~~\llr~~
\sum_{j=0}^\infty \frac{1}{p_j\pi_jQ_j(\eta)Q_{j+1}(\eta)}
\sum_{k=0}^j r_k\pi_kQ_k^2(\eta) =\infty. 
\end{equation}
\end{lemma}
\vspace{.2cm}

Defining $\tL_n$ in analogy with \eqref{Ln} we readily obtain
\[
\tL_\infty = \sum_{j=0}^\infty \frac{1}{p_j\pi_jQ_j(\eta)Q_{j+1}(\eta)}.
\]
So, in analogy with \eqref{recAA}, Lemma \ref{rho0} yields
\begin{equation}
\label{tLrho}
\mathcal{X} ~~\mbox{is~aperiodic~and~} \tL_\infty = \infty
~~\Ra~~ \lim_{n\to\infty}\frac{\rho_n(-\eta)}{\rho_n(\eta)} = 0.
\end{equation}
By \eqref{monQ} we have  $\tL_\infty\geq L_\infty$,
so the premise in \eqref{tLrho} certainly prevails if $\mathcal{X}$ is aperiodic
and recurrent. For later use we note that the condition $\tL_\infty = \infty$ has an
interpretation in terms of the measure $\psi$, namely, by \cite[Theorem 3.2]{DS95a},
\begin{equation}
\label{Thm3.2}
\tL_\infty = \infty ~~\llr~~ \int_{[-\eta,\eta]}\frac{\psi(dx)}{\eta - x} = \infty,
\end{equation}
so that in particular $\psi(\{\eta\} = 0$ if $\tL_\infty < \infty$.

Another sufficient condition for the left-hand side of \eqref{rho0eq} is
obtained in analogy with \eqref{rjpj}, namely
\begin{equation}
\label{rpQ}
\sum_{j=0}^\infty \frac{r_jQ_j(\eta)}{p_jQ_{j+1}(\eta)} = \infty ~~\Ra ~~
\lim_{n\to\infty}\frac{\rho_n(-\eta)}{\rho_n(\eta)}= 0.
\end{equation}
Note that by \eqref{monQ} we have
\begin{equation}
\label{rpr}
\sum_{j=0}^\infty \frac{r_jQ_j(\eta)}{p_jQ_{j+1}(\eta)} \geq
\sum_{j=0}^\infty \frac{r_j}{p_j},
\end{equation}
so that \eqref{rpQ} improves upon the sufficient condition implied by
\eqref{rjpj}, \eqref{tXrho} and \eqref{XtX}. 

The following is a sufficient condition for the left-hand side of \eqref{rho0eq}
in terms of the orthogonalizing measure $\psi$.
\begin{theorem}
\label{ze>-eta}
We have 
\[
\ze > -\eta ~~\Ra ~~ \lim_{n\to\infty}\frac{\rho_n(-\eta)}{\rho_n(\eta)}= 0.
\]
\end{theorem}
\vspace{.2cm}
\begin{proof}\rm
In view of \eqref{rpQ} and \eqref{rpr} it is no restriction to assume in the
remainder of this proof that $r_j\to 0$. Define the polynomials $S_n$ by
\[
\begin{array}{l}
xS_n(x)=q_nS_{n-1}(x)+p_nS_{n+1}(x),\quad n > 1,\\
S_0(x)=1,\quad  p_0S_1(x)=x,
\end{array}
\]
and let $\phi$ be the measure with respect to which these polynomials are
orthogonal. Then $\phi$ is symmetric about 0. Let $[-\theta, \theta]$ be the
smallest interval containing the support of $\phi$. By $J_\psi$ and $J_\phi$ we
denote the operators
\[
J_\psi a_n = q_na_{n-1}+r_na_n+p_na_{n+1} \mbox{~~and~~}
J_\phi a_n = q_na_{n-1}+p_na_{n+1}.
\]
The spectra of $J_\psi$ and $J_\phi$ on the space of square summable
sequences correspond to supp($\psi$) and supp($\phi$), respectively, and any
mass point of $\psi$ ($\phi$) is an eigenvalue of $J_\psi$ ($J_\phi$) (see, for
example, Van Assche \cite{V96} for these and subsequent results). Since
$r_j\to 0$ the difference $J_\psi-J_\phi$ is a compact operator, so, by Weyl's
theorem on bounded linear operators, supp($\psi$) and supp($\phi$) differ by at
most countably many points, each being a mass point of the corresponding measure.
Since $r_j \geq 0$ we also have $\ze\geq -\theta$ and $\eta\geq\theta$. (This
follows also from \cite[Theorems III.5.7 and IV.2.1]{C78}.) 
If $\eta>\theta$ then $\eta$ is a mass point of $\psi$ and, by \eqref{tLrho}
and \eqref{Thm3.2}, we are done. On the other hand, if $\eta=\theta$ then
$\ze>-\theta$, so that $-\theta$ is a mass point of $\phi$ and, by symmetry,
also $\theta=\eta$ is a mass point of $\phi$. It follows that
\begin{equation}
\label{phi}
\int_{[-\eta,\eta]}\frac{\phi(dx)}{\eta-x} = \infty.
\end{equation}
From \cite[Theorem IV.2.1]{C78} and \eqref{loweta} we know that the sequence 
\[
\left\{\frac{p_{n-1}q_n}{(\eta-r_{n-1})(\eta-r_n)}\right\}_n
\]
constitutes a {\em chain sequence\/}. Moreover, $\psi$ not being symmetric, we
have $r_j>0$ for some $j$, while
\[
\frac{p_{n-1}q_n}{\eta^2} \leq \frac{p_{n-1}q_n}{(\eta-r_{n-1})(\eta-r_n)},
\]
so that $\{p_{n-1}q_n/\eta^2\}_n$ constitutes a chain sequence that does not
determine its parameters uniquely. But this contradicts \eqref{phi}, by
\cite[Theorem 1]{S94}, so $\eta=\theta$ is not possible.
\end{proof}

\noindent
{\bf Remark}. An alternative proof involving a probabilistic argument is the
following. Define the polynomials $\tS_n$ by
\[
\begin{array}{l}
x\tS_n(x)=\tq_n\tS_{n-1}(x)+\tp_n\tS_{n+1}(x),\quad n > 1,\\
\tS_0(x)=1,\quad  \tp_0\tS_1(x)=x,
\end{array}
\]
with $\tp_n$ and $\tq_n$ as in \eqref{tbd}.
Since $\tp_j+\tq_j = 1-\tr_j \leq 1$, the polynomials $\tS_n$ correspond to a 
discrete-time birth-death process $\mathcal{Y}$ with an ignored state $\de$
that can be reached with probability $\tr_j$ from state $j \in \mathcal{N}$
(see \cite[Sect.~3]{CD06}).
Since $\tr_j > 0$ for at least one $j\in \mathcal{N}$, the process $\mathcal{Y}$
is transient and, as a consequence (see \cite[p.~70]{KM59}), the
(symmetric) measure $\tilde{\phi}$ associated with $\mathcal{Y}$ satisfies
\begin{equation}
\label{tphi}
\int_{[-1,1]}\frac{\tilde{\phi}(dx)}{1-x} < \infty.
\end{equation}
As before, let $[-\tilde\theta, \tilde\theta]$ be the smallest interval
containing the support of $\tilde{\phi}$. Now applying the argument involving
Weyl's theorem in the proof above to the operators $J_{\tilde{\psi}}$ and
$J_{\tilde{\phi}}$, the assumption $\tilde\theta = \tilde{\eta}~(=1)$ implies
$-\tilde\theta=-1<\ze/\eta=\tilde\ze$, so that $-1$, and hence, by symmetry,
$1$, is a mass point of $\tilde\phi$.  This, however, contradicts $\eqref{tphi}$.
On the other hand, the assumption $\tilde\theta < 1$ implies that $1$ is a mass
point of $\tilde{\psi}$, and hence $\eta$ a mass point of $\psi$, which, by
\eqref{tLrho} and \eqref{Thm3.2}, yields the result. \hfill $\Box$
\medskip

Our next step will be to study the asymptotic behaviour of 
$\rho_n(-\eta)/\rho_n(\eta)$ in the specific setting of Theorem \ref{Cw}. So we
will now assume that the random walk measure $\psi$ satisfies the conditions
$(i)$, $(ii)$ and $(iii)$ preceding Theorem \ref{Cw}, so that
supp($\psi$) $=[-\eta,\eta]$. In addition we will assume that $\psi$ is
{\em regular} in the sense of Ullman-Stahl-Totik (see Stahl and Totik
\cite[Def.~3.1.2]{ST92}), which amounts to assuming that
$\lim_{n\to\infty}\ga_n^{1/n}=2\eta$. (Recall that $\ga_n$ is the coefficient
of $x^n$ in $p_n(x)$.) Applying Theorem 1.2 of Danka and Totik \cite{DT18}
then leads to the conclusion that
\[
\lim_{n\to\infty} n^{2\al+2}\rho_n(\eta) = (2\eta)^{-\al-1}w(\eta-)\Ga(\al+1)\Ga(\al+2).
\]
By considering the measure with respect to which the polynomials $(-1)^nQ_n(-x)$
are orthogonal, one obtains in a similar way
\[
\lim_{n\to\infty} n^{2\be+2}\rho_n(-\eta) = (2\eta)^{-\be-1}w(-\eta+)\Ga(\be+1)\Ga(\be+2).
\]
From Theorem \ref{Cw} we know already that $0 < \al\leq\be$, so the preceding
limit results lead to the following theorem.
\begin{theorem}
\label{rhow}
If $\mathcal{X}$ is aperiodic, and  the corresponding measure $\psi$ is regular
and satisfies the conditions $(i)$, $(ii)$ and $(iii)$ preceding Theorem \ref{Cw},
then $0 < \al\leq\be$ and
\[
\lim_{n\to\infty} \frac{\rho_n(-\eta)}{\rho_n(\eta)} =
\left\{
\begin{array}{l@{}l}
0 &\mbox{if}~~ \al<\be\\
\ds\frac{w(-\eta+)}{w(\eta-)} \quad &\mbox{if}~~ \al=\be.
\end{array}
\right.
\]
\end{theorem}

We note again that $w(-\eta+)=0$ if $\ze>-\eta$, so the result is consistent with
Theorem \ref{ze>-eta}.

\section{Results}
\label{res}

In this section we will verify Conjecture \ref{c2} under mild regularity conditions
on the one-step transition probabilities of the process $\mathcal{X}$ and the
associated random walk measure $\psi$. Unless stated otherwise we will
assume $\mathcal{X}$, and hence $\mathcal{\tilde{X}}$, to be aperiodic, that
is, $r_j>0$ for at least one state $j\in \mathcal{N}$. We may further restrict our
analysis to the setting in which
\[
\sum_{j = 0}^\infty \frac{1}{p_j\pi_j} \sum_{k=0}^j r_k\pi_k<\infty
\mbox{~~and~~} \tL_\infty<\infty,
\]
since we know already by \eqref{Rrho}, \eqref{tLrho} and Theorem \ref{supCrho}
that the conjecture holds true in the opposite case, both sides of \eqref{con2}
then being equal to zero.
In view of \eqref{ineqr} we thus have $\sum r_n < \infty$, and hence $r_n\to 0$
as $n\to\infty$.

In what follows we denote the smallest and largest {\em limit\/} point of
supp($\psi$) by $\si$ and $\tau$, respectively. Evidently, $\ze\leq\si\leq\tau\leq\eta$.
The next lemma shows that we can draw some useful conclusions on the measure
$\psi$ if, besides $\tL_\infty<\infty$ and $r_n\to 0$, the product $p_{n-1}q_n$
tends to a limit as $n\to\infty$. 

\begin{lemma}
\label{blum}
Let $\ds\lim_{n\to\infty}r_n = 0$ and $\tL_\infty<\infty$. If $\ds\lim_{n\to\infty}
p_{n-1}q_n = \be$, then $\eta=\tau=2\sqrt{\be}>0$ and $\ze=\si=-2\sqrt{\be}$.
\end{lemma} 
\begin{proof}\rm
The {\it monic} polynomials $P_n=p_0\dots p_{n-1}Q_n$ satisfy the recurrence
\[
\begin{array}{l}
P_{n+1}(x)=(x-r_n)P_n(x)-p_{n-1}q_nP_{n-1}(x), \quad n > 0,\\
P_0(x)=1,\quad  P_1(x)=x-r_0.
\end{array}
\]
By Blumenthal's theorem (see Chihara \cite{C68}) we have $\si=-\tau=-2\sqrt{\be}$
when $r_n\to 0$ and $p_{n-1}q_n\to\be$ as $n\to\infty$.
If $\eta>\tau$ then $\eta$ must be an isolated point of supp($\psi$), and hence
$\psi(\{\eta\})>0$. But in view of \eqref{Thm3.2} this would contradict our
assumption $\tL_\infty<\infty$, so we must have $\eta=\tau = 2\sqrt{\be}$
and hence $\be>0$, by \eqref{loweta}. Finally, by \eqref{infsupp},
$\ze\geq -\eta$, but since $\ze\leq\si=-\eta$, we must have $\ze=\si$.
\end{proof}

Note that, as a consequence of this lemma, Theorem \ref{ze>-eta} is of no use
to us in verifying Conjecture \ref{c2} when $p_{n-1}q_n$ tends to limit, for in
that case $\ze>-\eta$ can only occur if $\tL_\infty=\infty$ or $r_n\not\to 0$.

Regarding the parameters $p_j$ and $q_j$ we will now impose the condition
\begin{equation}
\label{ass1}
\sum _{j=1}^\infty |p_jq_{j+1}-p_{j-1}q_j|<\infty,
\end{equation}
implying in particular that $p_nq_{n+1}$ tends to a limit. We will further assume
\begin{equation}
\label{1/4}
\lim_{n\to\infty}p_nq_{n+1} = \frac{1}{4},
\end{equation}
so that, by the
previous lemma, $\ze=\si=-1$ and $\eta=\tau=1$. The latter assumption entails
no loss of generality, since, in view of \eqref{trho} and \eqref{tC}, verifying
Conjecture \ref{c2} is equivalent to verifying a similar conjecture in terms of
$\mathcal{\tilde{X}}$, while by \eqref{tbd} and the previous lemma,
\[
\tp_n\tq_{n+1} = \frac{p_nq_{n+1}}{\eta^2} \to \frac{\be}{\eta^2} = \frac{1}{4}
\mbox{~~as~~} n\to\infty.
\]
Letting $\Psi$ as in \eqref{Psi} we can now invoke a theorem of M\'at\'e and Nevai
\cite{MN83} stating that $\Psi$ is continuously differentiable in $(-1,1)$ and
$\Psi'(x)>0$ for $x\in (-1,1)$, so that supp($\psi$) $= [-1,1]$. In view of
\eqref{gamma} and \eqref{1/4} we also have $\lim_{n\to\infty} \ga_n^{1/n} = 2$,
so that $\psi$ is regular in the sense of Ullman-Stahl-Totik.

In what follows we will assume that the limits $\Psi'(-1+)$ and  $\Psi'(1-)$ exist.
Recalling our earlier assumptions that $\mathcal{X}$ is aperiodic and $\tL_\infty < \infty$,
we now have, by \eqref{whit} and \eqref{Thm3.2}, not
only $\Psi(-1+) = \Psi(-1) = 0$ and $\Psi(1-) = \Psi(1)=1$ (implying the continuity of
$\Psi$), but also $\Psi'(-1+) = \Psi'(1-) = 0$, which implies the continuity of $\Psi'$
on $\mathbb{R}$. Next defining $\al$, $\be$ and $w$ as in \eqref{albe} and
\eqref{w}, the Theorems \ref{Cw} and \ref{rhow} lead to the conclusion that, under
the preceding conditions and if $0<w(1-)<\infty$, we have $0<\al\leq\be$ and
\begin{equation}
\label{wCw}
\lim_{n\to\infty} C_n(\psi) = \lim_{n\to\infty} \frac{\rho_n(-1)}{\rho_n(1)} =
\left\{
\begin{array}{l@{}l}
0 &\mbox{if}~~ \al<\be\\
\ds\frac{w(-1+)}{w(1-)} \quad &\mbox{if}~~ \al=\be.
\end{array}
\right.
\end{equation}

Collecting all our results we can now establish the following theorem, which
amounts to validity of Conjecture \ref{c2} under mild regularity conditions.
\begin{theorem}
\label{main}
Let $\mathcal{X}$ be a birth-death process with corresponding random walk measure
$\psi$, and let $\Psi$, $\al$, $\be$ and $w$ be defined as in \eqref{Psi},\eqref{albe}
and \eqref{w}.\\
$(i)$ If $\mathcal{X}$ is periodic, then
\[
\lim_{n\to\infty} C_n(\psi) = \lim_{n\to\infty} \frac{\rho_n(-\eta)}{\rho_n(\eta)} = 1.
\]
$(ii)$ If  $\mathcal{X}$ is aperiodic and
\begin{equation}
\label{ass2}
\sum_{j=0}^\infty \frac{1}{p_j\pi_j} \sum_{k=0}^j r_k\pi_k = \infty
\mbox{~~or~~}  
\sum_{j=0}^\infty \frac{1}{p_j\pi_jQ_j(\eta)Q_{j+1}(\eta)} = \infty,
\end{equation}
then
\[
\lim_{n\to\infty} C_n(\psi) = \lim_{n\to\infty} \frac{\rho_n(-\eta)}{\rho_n(\eta)} = 0.
\]
$(iii)$ If $\mathcal{X}$ is aperiodic, \eqref{ass2} does not hold (so that $r_n\to 0$),
and in addition,\\
\hspace*{0.4cm}$(a)$ the one-step transition probabilities of $\mathcal{X}$
satisfy $\sum _{j=1}^\infty |p_jq_{j+1}-p_{j-1}q_j|<\infty$,\\
\hspace*{0.4cm}$(b)$ the limits $\Psi'(-\eta+)$ and  $\Psi'(\eta-)$ exist,\\
\hspace*{0.4cm}$(c)$ the quantities $\al$ and $\be$ are finite,\\
\hspace*{0.4cm}$(d)$ the limits $w(-\eta+)$ and $w(\eta-)$ exist and are finite, and
$w(\eta-)>0$,\\ then $0<\al\leq\be$ and \vspace*{-0.5cm}
\begin{equation}
\label{wCweta}
\lim_{n\to\infty} C_n(\psi) = \lim_{n\to\infty} \frac{\rho_n(-\eta)}{\rho_n(\eta)} = 
\left\{
\begin{array}{l@{}l}
0 &\mbox{if}~~ \al<\be\\
\ds\frac{w(-\eta+)}{w(\eta-)} \quad &\mbox{if}~~ \al=\be.
\end{array}
\right.
\end{equation}
\end{theorem} 
\begin{proof}\rm
The first statement is implied by the fact that $\psi$ is symmetric if $\mathcal{X}$
is periodic, while the second statement follows from \eqref{Rrho}, \eqref{tLrho} and
Theorem \ref{supCrho}.
To prove the third statement we apply to the normalized version $\tilde{\mathcal{X}}$
of $\mathcal{X}$ the argument preceding this theorem. Obviously,
$\tilde{\Psi}'(x) = \eta\Psi'(\eta x)$ and $\tilde{w}(x) = \eta^{\al+\be+1}w(\eta x)$,
so subsequently rephrasing, with the help of \eqref{trho} and \eqref{tC}, conclusion
\eqref{wCw} and the conditions preceding it in terms of the original process
$\mathcal{X}$, gives us \eqref{wCweta}.
\end{proof}

\section{Concluding remarks}
\label{conc}

The previous analysis remains largely valid if we allow $p_j+q_j+r_j \leq 1$ and
interpret $\ka_j:=1-p_j-q_j-r_j$ as the {\em killing probability} of $\mathcal{X}$
in state $j$, that is, the probability of absorption into an (ignored) cemetary state
$\pa$, say. Karlin and McGregor's representation formula \eqref{repP} still holds
in this more general setting, but if $\ka_j>0$ for at least one state $j\in S$ (so
that $\pa$ is accessible from $\mathcal{N}$) we have to make some adjustments
to the preceding analysis.

First, asymptotic aperiodicity is not defined for $\mathcal{X}$ in this case, but
since the normalization \eqref{tbd} results in a process $\mathcal{\tilde{X}}$
which, as before, satisfies $\tp_j+\tq_j+\tr_j=1$ for all $j\in\mathcal{N}$, the
content of Subsection \ref{AA} remains relevant if $\mathcal{X}$ is replaced by
$\mathcal{\tilde{X}}$ (which will be different from $\mathcal{X}$, also if $\eta
= 1$.) Then, from \cite[Eq. (25)]{CD06} we know that
\[
Q_{n+1}(1) = 1 + \sum_{j=0}^n\frac{1}{p_j\pi_j} \sum_{k=0}^j \ka_k\pi_kQ_k(1),
\quad n\geq 0,
\]
so that $Q_{n+1}(1)\geq Q_n(1)$ with strict inequality for $n$ sufficiently large.
So we no longer have $Q_n(1)=1$ and therefore cannot assume the validity
of \eqref{monQ} and its consequence \eqref{rpr}. Note that
\begin{equation}
\label{Qsum}
\lim_{n\to\infty} Q_n(1) = \infty ~~\llr~~
\sum_{j = 0}^\infty \frac{1}{p_j\pi_j} \sum_{k=0}^j \ka_k\pi_k = \infty,
\end{equation}
while \cite[Theorem 5]{CD06} tells us that $\tau_j$, the probability of
eventual absorption at $\pa$ from state $j$, is given by
\[
\tau_j = 1 - \frac{Q_j(1)}{Q_\infty(1)}, \quad j\in\mathcal{N}.
\]
So eventual absorption at $\pa$ is certain if and only if $\lim_{n\to\infty}
Q_n(1)=\infty$.

It is easily seen that \cite[Lemma 2.1]{D18a}, and hence \eqref{RQ}, remain valid
in the more general setting at hand, but that is not so obvious for \eqref{QQ}.
In fact, it may be shown that \eqref{QQ} should be replaced by
\begin{equation}
\label{QQb}
\lim_{n\to\infty} \left|\frac{Q_n(1)}{Q_n(-1)}\right| = 0
~~\Ra~~ \lim_{n\to\infty} \left|\frac{Q_n(\eta)}{Q_n(-\eta)}\right| = 0,
\end{equation}
and so the
conclusion \eqref{Rrho} cannot be maintained. However, in view of \eqref{Qsum},
we may replace \eqref{Rrho} by
\begin{equation}
\label{RKrho}
\sum_{j=0}^\infty \frac{1}{p_j\pi_j} \sum_{k=0}^j r_k\pi_k =\infty ~~\mbox{and}~~
\sum_{j=0}^\infty \frac{1}{p_j\pi_j} \sum_{k=0}^j \ka_k\pi_k <\infty
~~\Ra~~ \lim_{n\to\infty} \frac{\rho_n(-\eta)}{\rho_n(\eta)}= 0.
\end{equation}
In other words, \eqref{Rrho} remains valid if we add the condition that 
absorption at $\pa$ is {\em not} certain.
This has consequences for Theorem \ref{main}, where the first condition in
\eqref{ass2} should be replaced by the two conditions in
\eqref{RKrho}.

All other results remain valid.

\vspace{1.0cm}
\noindent
{\Large{\bf Acknowledgement}}

\bigskip
\noindent
The authors thank Vilmos Totik for helpful comments and suggestions.


\end{document}